\documentclass[12pt,oneside]{amsart}

\usepackage{amssymb}
\usepackage{rotating}

\usepackage{amsfonts}
\usepackage{dsfont}
\usepackage{amsxtra}

\usepackage{index}

\theoremstyle{plain}
\newtheorem{Theo}{Theorem}[section]
\newtheorem{Lem}[Theo]{Lemma}
\newtheorem{Prop}[Theo]{Proposition}
\newtheorem{Cor}[Theo]{Corollary}

\theoremstyle{definition}
\newtheorem{Defi}[Theo]{Definition}

\theoremstyle{remark}
\newtheorem{Rem}[Theo]{Remark}

      \makeatletter
      \def\@setcopyright{}
      \def\serieslogo@{}
      \makeatother

\newcommand{\oo}{\mathcal{O}}

\begin{document}

\author{Julio Jos\'e Moyano-Fern\'andez}
\address{Institut f\"ur Mathematik, Universit\"at Osnabr\"uck. Albrechtstra\ss e 28a, D-49076 Osnabr\"uck, Germany}
\email{jmoyanof@uni-osnabrueck.de}

\title[On curve singularity invariants and reductions]{On some curve singularity \\ invariants and reductions}

\begin{abstract}
This paper deals with the study of the behaviour of the value semigroup of a curve singularity define over a global field reduced modulo a maximal ideal. We also define a global zeta function of the curve by means of motivic integration over a suitable ring of ad\'eles, whose reduction modulo a maximal ideal will coincide with already known zeta functions of the singularity.
\end{abstract}



 \thanks{The author was partially supported by the Spanish Government Ministerio de Educaci\'on y Ciencia (MEC)
grant MTM2007-64704 in cooperation with the European Union in the framework of the founds ``FEDER'', and by the Deutsche 
Forschungsgemeinschaft (DFG)}

\maketitle

\section{Preliminaries}

Resolution of singularities is a well-known topic since the pioneers Brill and Noether solved the case of complex curves more than a century ago. It became a central problem in algebraic geometry when mathematicians asked for generalisations in the two usual ways, namely resolution for varieties in higher dimensions and resolution in arbitrary ground fields. The introduction of basis change methods in scheme theory allowed to define the concept of reduction modulo a maximal ideal of an algebraic variety defined over a number field, which is nothing else than an attempt to preserve as many good properties as possible when one extend the variety to a scheme over the ring of integers of the field the curve is defined over.
\medskip

Let $C$ be a complex curve singularity. Many equivalent invariants, both topological and analytic,  are since a long time well understood.   We remark among others the associated value semigroup and the Poincar\'e series, profusely studied by v.g. Bayer \cite{B}, Garc\'\i a \cite{G}, Kunz \cite{K}, Waldi \cite{W}, and Campillo, Delgado and Gusein-Zade \cite{C}--\cite{CDG4}, \cite{Del2}, as well as by St\"ohr \cite{Stohr}--\cite{stohrneu} and Z\'{u}\~{n}iga Galindo \cite{Z1}, \cite{Z3} in the last years. The question this note deals with, is to describe the behaviour of these two singularity invariants under reduction modulo a maximal ideal $\mathfrak{P}$ when the singularity is defined over a number field. The first issue to do is to understand resolution of singularities under such a reduction.  
\medskip

For the convenience of the reader we repeat the relevant material from \cite{D}, thus making our exposition self-contained.

\subsection{Reduction modulo $\mathfrak{P}$}

Let $K$ be an arbitrary field of characteristic $0$. Let $R$ be a
discrete valuation ring with quotient field $K$ and maximal ideal
$\mathfrak{P}$.
\medskip

Let $C$ be a plane curve defined over the field $K$ given by an equation
$f(x_1,x_2) \in K [x_1,x_2]$. Let us consider the schemes $X =
\mathrm{Spec}~K[x_1,x_2]$,
$\widetilde{X}=\mathrm{Spec}~R[x_1,x_2]$,
$\overline{X}=\mathrm{Spec}~\overline{K}[x_1,x_2]$ and
\[
D=\mathrm{Spec} \left ( K[x_1,x_2] / (f(x_1,x_2)) \right )
\subset X.
\]

Let $Z$ be a closed subscheme of $\mathbb{P}_{X}^{k}:=
\mathrm{Proj~} K[x_1, x_2, \ldots , x_k]$ for some nonnegative integer $k>0$. We will
denote by $\widetilde{Z}$ the scheme-theoretic closure of $Z$ in
$\mathbb{P}_{\widetilde{X}}^{k}$. If $Z$ is reduced (resp.
integral), then $\widetilde{Z}$ is reduced (resp. integral).
Moreover, $\widetilde{Z}$ is flat over $\mathrm{Spec~}R$.

\begin{Defi}
The reduction modulo $\mathfrak{P}$ of $Z$ is the scheme
\[
\overline{Z} = Z \mathrm{~mod~} \mathfrak{P} := \widetilde{Z}
\times_R \mathrm{Spec~} \overline{K}.
\]
\end{Defi}
Notice that $Z \mathrm{~mod~} \mathfrak{P}$ can be considered a
closed subscheme of $\widetilde{Z}$.
\medskip

\subsection{Resolution of singularities}

Let us recall the basics on resolution of curve singularities. By the sack of simplicity we will assume the singularities to be totally rational (i.e., all field extension degrees involved are assumed to be $1$).

\begin{Defi}
A resolution for the curve $C$ over $K$ is a pair $(Y,\pi)$, where
$Y$ is a closed integral scheme of $\mathbb{P}_{X}^{k}$ for some nonnegative integer
$k$ and $\pi: Y \to X$ is a morphism such that
\begin{itemize}
    \item[(i)] $Y$ is smooth over $\mathrm{Spec~} K$;
    \item[(ii)] the restriction $\pi: Y \setminus \pi^{-1}(D) \to X \setminus
    D$ is an isomorphism;
    \item[(iii)] the reduced scheme $(\pi^{-1}(D))_{\mathrm{red}}$
    associated to $\pi^{-1}(D)$ has only normal crossings as
    subscheme of $Y$.
\end{itemize}
\end{Defi}

The scheme $(\pi^{-1}(D))_{\mathrm{red}}$ will be called the
exceptional divisor of $\pi$. Its irreducible components will be
denoted by $E_i$, for $i \in T$.
\medskip

Next result is well-known in singularity theory:

\begin{Theo}
Any plane curve singularity has a resolution.
\end{Theo}

\begin{Defi}[\cite{D}] \label{def:resgoodreduction}
We say that $(Y,\pi)$ is a resolution for $C$ over $K$ with good
resolution modulo $\mathfrak{P}$ if $(Y,\pi)$ is a resolution for
$C$ over $K$ satisfying
\begin{itemize}
    \item[(i)] $\overline{Y}$ is smooth over $\mathrm{Spec~} \overline{K}$;
    \item[(ii)] $\overline{E}_i$ is smooth over $\mathrm{Spec~}
    \overline{K}$ for each $i \in T$, and $\bigcup_{i \in T}
    \overline{E}_i$ has only normal crossings as a subscheme of
    $\overline{Y}$;
    \item[(iii)] $\overline{E}_i$ and $\overline{E}_j$ have no
    common irreducible components for $i \ne j$.
\end{itemize}
\end{Defi}

\begin{Rem} \label{remark:T}
Note that ``reduction modulo $\mathfrak{P}$'' does not change the
set $T$, that is, it does not change the number of components of
the exceptional divisor.
\end{Rem}

If $C$ is a plane curve having a closed singular point, we can
reach a resolution for $C$ just blowing-up closed points
successively to get a chain
\begin{equation}  \nonumber
Y:=X_{N} \overset{\pi_{N}}{\longrightarrow} X_{N-1}
\overset{\pi_{N-1}}{\longrightarrow} \ldots \longrightarrow X_3
\overset{\pi_3}{\longrightarrow} X_2
\overset{\pi_2}{\longrightarrow} X_1
\overset{\pi_1}{\longrightarrow} X_0=:X,
\end{equation}
where $\pi_i$ is the blowing-up of a closed point $p_{i-1} \in
X_{i-1}$, for $1 \le i \le N$. Then $Y$ is achieved by blowing-up
successively closed points $p_i \in X_{i}$, $0 \le i \le N-1$.
This process is called the \emph{resolution process} of $C$. Every
such a closed point $p_i$ corresponds to a local ring $R_i :=
\mathcal{O}_{X_{i},p_i}$. Since the curve $C$ is assumed to be
totally rational, the degree of this field extension is equal to
$1$. Finally, set $\pi := \pi_{N} \circ \pi_{N-1} \circ \ldots
\circ \pi_2 \circ \pi_1$. Further details on curve singularities over arbitrary fields can be founded in \cite{kiyek};  for another recent account see \cite{MoCu}.
\medskip

On the other hand, let us take $\pi:\widetilde{X}_{i} \to
\widetilde{X}_{i-1}$ for any $i \in \{1, \ldots, N \}$. Then we
can consider the restriction to $\widetilde{X}_{i}$ of the
projection $\mathbb{P}^k_{\widetilde{X}} \to \widetilde{X}_{i-1}$
for some $k > 0$, namely
\[
\widetilde{\pi}_i:\widetilde{X}_{i} \to \widetilde{X}_{i-1}.
\]
Even more, using base extension we can get morphisms
\[
\overline{\pi}_i:\overline{X}_{i} \to \overline{X}_{i-1}
\]
with $1 \le i \le N$. Set $\overline{\pi} := \overline{\pi}_{N}
\circ \overline{\pi}_{N-1} \circ \ldots \circ \overline{\pi}_2
\circ \overline{\pi}_1$.
\medskip

\begin{Theo} [\cite{D}]\label{thm:uno}
If $(Y,\pi)$ is a resolution for $C$ over $K$ with good reduction
modulo $\mathfrak{P}$, then $(\overline{Y},\overline{\pi})$ is a
resolution for $C$ modulo $\mathfrak{P}$ over $\overline{K}$.
\end{Theo}

\begin{proof} First of all, $\overline{Y}$ is a closed integral scheme of
$\mathrm{Proj~} \overline{K}[x_1, \ldots , x_k]$ for some $k >0$
(see \cite{D}, Proposition 2.6 (b)). By Definition
\ref{def:resgoodreduction}, condition (ii), $\overline{Y}$ is
smooth over $\mathrm{Spec}~ \overline{K}$. Considering now the
reduced structure of the exceptional divisor of $\overline{\pi}$,
i.e., $\bigcup_{i \in T} \overline{E}_i$, it has only normal
crossings as subscheme of $\overline{Y}$ (again by Definition
\ref{def:resgoodreduction}). Last, the restriction $\overline{Y}
\setminus \bigcup_{i \in T} \overline{E}_i \to \overline{X}
\setminus D$ is an isomorphism, since the morphism $\overline{Y}
\to \overline{X}$ is birational (cf. \cite{D}, Proposition 2.6
(i)). Then $(\overline{Y},\overline{\pi})$ is a resolution for $C
\mathrm{~mod~} \mathfrak{P}$.
\end{proof}

\begin{Defi}
Let $C$ (resp. $C^{\prime}$) be a plane curve with resolution
$(Y,\pi)$ (resp. $(Y^{\prime},\pi^{\prime})$) over a field $K$
(resp. $K^{\prime}$). Let
\begin{equation} \nonumber
Y:=X_{N} \overset{\pi_{N}}{\longrightarrow} X_{N-1}
\overset{\pi_{N-1}}{\longrightarrow} \ldots \longrightarrow X_3
\overset{\pi_3}{\longrightarrow} X_2
\overset{\pi_2}{\longrightarrow} X_1
\overset{\pi_1}{\longrightarrow} X_0=:X,
\end{equation}
be the resolution process for $C$, and let
\begin{equation} \nonumber
Y^{\prime}:=X^{\prime}_{N^{\prime}}
\overset{\pi^{\prime}_{N^{\prime}}}{\longrightarrow}
X^{\prime}_{N^{\prime}-1}
\overset{\pi^{\prime}_{N^{\prime}-1}}{\longrightarrow} \ldots
\longrightarrow X^{\prime}_3
\overset{\pi^{\prime}_3}{\longrightarrow} X^{\prime}_2
\overset{\pi^{\prime}_2}{\longrightarrow} X^{\prime}_1
\overset{\pi^{\prime}_1}{\longrightarrow}
X^{\prime}_0=:X^{\prime},
\end{equation}
be the resolution process for $C^{\prime}$. The curves $C$ and
$C^{\prime}$ are said to have the same resolution process if
$N=N^{\prime}$ and the strict transforms $\widetilde{C}^{(i)}$
(resp. $(\widetilde{C^{\prime}})^{(i)}$) of $C$ (resp.
$C^{\prime}$) via $\pi_i$ (resp. $\pi^{\prime}_i$) have the same
multiplicity.
\end{Defi}

\begin{Prop} \label{prop:agreeprocess}
Let $C$ be a plane curve. Let $(Y,\pi)$ be the resolution process
of a totally rational closed point of $C$, and
$(\overline{Y},\overline{\pi})$ the corresponding process of the
curve $\overline{C}$. These two resolution processes agree.
\end{Prop}

\begin{proof}~By construction of $\overline{\pi}$ and Remark \ref{remark:T}, both $(Y,\pi)$ and
$(\overline{Y},\overline{\pi})$ consist of the same number of
blowing-ups. Moreover, since the curve $C$ is assumed to be
totally rational, the multiplicities of the strict transforms of
$C$ and $\overline{C}$ coincide.
\end{proof}

The following theorem, due to Denef (cf. \cite[Theorem 2.4]{D}), ensures the existence of a resolution
with good reduction modulo $\mathfrak{P}$ for plane curves over number fields.

\begin{Theo} \label{thm:denef}
Let $F$ be a number field, $A$ its ring of algebraic integers. Let
$C$ be a plane curve over $F$ and $(Y,\pi)$ a resolution for $C$
over $F$. Then $(Y,\pi)$ is a resolution for $C$ over $F$ with
good reduction modulo $\mathfrak{P} A_{\mathfrak{P}}$ for all
except a finite number of maximal ideals $\mathfrak{P}$ of $A$.
\end{Theo}

\section{The value semigroup and Poincar\'e series of a plane curve singularity and reduction modulo $\mathfrak{P}$}

After introducing some elementary machinery, we will prove a first remarkable result: the value semigroup does not change under reduction modulo $\mathfrak{P}$.
\medskip

Let $F$ be a number field. Let $A$ be the ring of integers in $F$.
Assume that $C$ is a complete, geometrically irreducible,
algebraic curve defined over $F$. Let $P$ be a rational and
singular point of $C$. Assume that $\widehat{\mathcal{O}}_{P,C}$
is totally rational. Then $\widehat{\mathcal{O}}_{P,C}$ \ can be
presented in the form
\[
\widehat{\mathcal{O}}_{P,C}:=\left\{
\left(
{\textstyle\sum\nolimits_{i=0}^{\infty}}
a_{i,1}t_{1}^{i},\ldots,%
{\textstyle\sum\nolimits_{i=0}^{\infty}}
a_{i,d}t_{d}^{i}\right)
\in\widetilde{\mathcal{O}}_{P,C} \mid \Delta =0\right\}  ,
\]
 where$\widetilde{\mathcal{O}}_{P,C}\cong F\left[ \! \left[  t_{1}\right] \! \right]
\times\ldots\times F\left[  \! \left[ t_{d}\right]  \! \right]  $ and
$\Delta=0$ is a homogeneous system of linear equations with
coefficients in $F$.

\subsection{Value semigroup}

Every factor in $\widetilde{\mathcal{O}}_{P,C}$ yields a discrete valuation $v_i$, for every $1 \le i \le d$, and thus a vector $\underline{v}(\underline{z})=(v_1{z_1}, \ldots , v_d(z_d))$ for any nonzero divisor $\underline{z}=(z_1, \ldots , z_d) \in \widetilde{\mathcal{O}}_{P,C}$. 

\begin{Defi}
The value semigroup $S$ of $\widehat{\oo}_{P,C}$ consists of all the elements of the form $\underline{v}(\underline{z})=(v_1(z_1), \ldots , v_d(z_d)) \in \mathbb{N}^d$ for all the nonzero divisors $\underline{z} \in \widetilde{\mathcal{O}}_{P,C}$. 
\end{Defi}

Observe that the value semigroup of $\widehat{O}_{P,C}$ coincides with that of $\mathcal{\oo}_{P,C}$ (cf. \cite[\S 2]{MZ}).
\medskip

The \textit{degree of singularity} of the ring $\oo_{P,C}$ is defined as
\[
\delta(\oo_{P,C})=\delta_P=\dim_{k}\widetilde{\mathcal{O}}_{P,C}/\mathcal{O}_{P,C},
\]
where $\widetilde{\mathcal{O}}_{P,C}$ is the normalisation of the ring $\mathcal{O}%
_{P,C}$ . By \cite[Theorem 1]{Ros1}, $\delta_{P}<\infty$. 
\medskip

We set $\underline{1}:=\left(  1,\ldots,1\right)  \in\mathbb{N}^{d}$ and, for
$\underline{n}=\left(  n_{1},\ldots,n_{d}\right)  \in\mathbb{N}^{d}$, we consider the vector
$\left\Vert \underline{n}\right\Vert :=n_{1}+\ldots+n_{d}$. We introduce a
partial order in $\mathbb{N}^{d}$, \textit{the product order}, by taking
$\underline{n}\geq\underline{m}$, if $n_{i}\geq m_{i}$ for every $i=1,\ldots,d$.

\begin{Defi}
Let $C$ be a plane curve. We define a complex model for $C$ to be
a plane curve $C^{\prime}$ over $\mathbb{C}$ having the same
resolution process.
\end{Defi}

The following theorem is due to Campillo (see \cite[Proposition
4.3.12]{C}):

\begin{Theo} \label{thm:campillo}
Two given plane curves, defined over fields which are different in
general, have the same value semigroup if and only if they have a
complex model in common. In particular, the value semigroup of a
curve and that of any of its models agree.
\end{Theo}


Let $\mathfrak{P}$ be a maximal \ ideal of $A$ with residue field
$\mathbb{F}_{q}$. We define the reduction mod $\mathfrak{P}$ of
$\widehat{\mathcal{O}}_{P,C}$ as
\[
\widehat{\mathcal{O}}_{P,C\text{ mod }\mathfrak{P}}:=\left\{
\left(
{\textstyle\sum\nolimits_{i=0}^{\infty}}
b_{i,1}t_{1}^{i},\ldots,%
{\textstyle\sum\nolimits_{i=0}^{\infty}}
b_{i,d}t_{d}^{i}\right)
\in\widetilde{\mathcal{O}}_{P,C\text{ mod } \mathfrak{P}}\mid\overline{\Delta }=0\right\}  ,
\]
where $\widetilde{\mathcal{O}}_{P,C \text{ mod } \mathfrak{P}}\cong\mathbb{F}_{q}\left[ \!
\left[ t_{1}\right] \!  \right]
\times\ldots\times\mathbb{F}_{q}\left[  \! \left[ t_{d}\right] \!
\right]  $ and $\overline{\Delta}=0$\ denotes the reduction mod
$\mathfrak{P}$ of $\Delta=0$.
\medskip

We can now compare the value semigroup of a curve over a number
field of characteristic $0$ and its reduction modulo
$\mathfrak{P}$:

\begin{Theo} \label{lemma:semigroups}
For all except a finite number of maximal ideals $\mathfrak{P}$ of
$A$ one has
\[
S(\mathcal{O}_{P,C}) = S(\mathcal{O}_{P, C \mathrm{~mod~}
\mathfrak{P}}).
\]
\end{Theo}

\begin{proof}
It is just to apply previous results:

\begin{enumerate}
    \item We have a field number of characteristic $0$, then we
    have a resolution $(Y,\pi)$ of $C$ over $F$ with good
    reduction modulo $\mathfrak{P}$ for almost all $\mathfrak{P}$,
    by Theorem \ref{thm:denef}.
    \item Then $(\overline{Y},\overline{\pi})$ is a resolution of
    $\overline{C}=C \mathrm{~mod~} \mathfrak{P}$ over
    $\overline{K}$ for almost all $\mathfrak{P}$, by Theorem \ref{thm:uno}, and with the same
    resolution process as $(Y,\pi)$ by Proposition
    \ref{prop:agreeprocess}.
    \item Then $C$ and $\overline{C}$ have the same complex model
    for almost all $\mathfrak{P}$.
    \item Last, by Theorem \ref{thm:campillo} the curves $C$ and
    $\overline{C}$ have the same value semigroup, for almost all
    $\mathfrak{P}$.
\end{enumerate}
\end{proof}

\subsection{Generalised Poincar\'e series and Zeta functions}

Let $Var_{k}$ denote the category of $k$-algebraic varieties, and by
$K_{0}\left(  Var_{k}\right)  $\ the corresponding Grothendieck ring. It is
the ring generated by symbols $\left[  V\right]  $, for $V$ an algebraic
variety, with the relations $\left[  V\right]  =\left[  W\right]  $ if $V$ is
isomorphic to $W$, $\left[  V\right]  =\left[  V\setminus Z\right]  +\left[
Z\right]  $ if $Z$ is closed in $V$, and $\left[  V\times W\right]  =\left[
V\right]  \left[  W\right]  $. We denote $\boldsymbol{1}:=\left[
\text{point}\right]  $, $\mathbb{L}:=\left[  \mathbb{A}_{k}^{1}\right]  $ and
$\mathcal{M}_{k}:=K_{0}\left(  Var_{k}\right)  $\ $\left[  \mathbb{L}%
^{-1}\right]  $ the ring obtained by localization with respect to the
multiplicative set generated by $\mathbb{L}$. \medskip

For $\underline{n}\in S$ we set
\[
\mathcal{I}_{\underline{n}}:=\left\{  I\subseteq\mathcal{O}_{P,C}\mid I=z\mathcal{O}_{P,C}%
\text{, with }\underline{v}(\underline{z})=\underline{n}\right\}  ,
\]
and for $m\in\mathbb{N}$,%
\[
\mathcal{I}_{m}:=%
{\textstyle\bigcup\nolimits_{\substack{\underline{n}\in S\\\left\Vert
\underline{n}\right\Vert =m}}}
\mathcal{I}_{\underline{n}}.
\]

\begin{Defi}
\label{def5}We associate to $\mathcal{O}_{P,C}$ the two following zeta functions:
\begin{equation}
Z\left( t_{1},\ldots,t_{d},\mathcal{O}_{P,C}\right)  :=%
{\textstyle\sum\nolimits_{\underline{n}\in S}}
\left[  \mathcal{I}_{\underline{n}}\right]  \mathbb{L}^{-\left\Vert
\underline{n}\right\Vert }t^{\underline{n}}\in\mathcal{M}_{k}\left[ \! \left[
t_{1},\ldots,t_{d}\right] \! \right]  , \label{zeta1}%
\end{equation}
where $t^{\underline{n}}:=t_{1}^{n_{1}}\cdot\ldots\cdot t_{d}^{n_{d}}$, and
\begin{equation}
Z\left(t,\mathcal{O}_{P,C}\right)  :=Z\left (t,\ldots,t,\mathcal{O}_{P,C}\right)  .
\label{zeta2}%
\end{equation}
\end{Defi}

This zeta function was introduced in \cite{MZ}; it coincides--up to a factor--with the generalised Poincar\'e series $P_g(t_1, \ldots, t_d)$ defined by Campillo, Delgado and Gusein-Zade in \cite{CDG4} and was studied under a more general setting by the author in \cite{MM}:

\begin{Lem} \label{lemma:coincidence}
\[
Z(t_1, \ldots , t_d)=\mathbb{L}^{\delta +1} P_g(t_1, \ldots , t_d).
\]
\end{Lem}


The generalised Poincar\'e series associated to a ring $\oo_{P,C}$ only depends on the value semigroup of $\oo$, hence by Theorem \ref{lemma:semigroups} we have

\begin{Cor}
\[
Z(t_1, \ldots , t_d;\oo_{P,C}) = Z(t_1, \ldots , t_d; \oo_{P, C \mathrm{~mod~} \mathfrak{P}})
\]
\[
P_g(t_1, \ldots , t_d; \oo_{P,C}) = P_g(t_1, \ldots , t_d; \oo_{P, C \mathrm{~mod~} \mathfrak{P}})
\]
\end{Cor}

\section{Motivic integration over the ring of Ad\`eles}

The goal of this section is to define a global zeta function on the curve by using the theory of Ad\`eles such that it can be expressed---roughly speaking---as a product of local zeta functions (Poincar\'e series), from which the ones corresponding to singular points are non-trivial. A good general reference here is the book of Cassels and Fr\"olich \cite{CF}.
\medskip

Let $X$ be a curve defined over $k$. Consider the family of
local rings $\{\mathcal{O}_{P,X}\}_{P \in X}$. From now on we will write $\mathcal{O}_P$ instead of $\mathcal{O}_{P,X}$. If $k(P)$ is the residue
field at each point $P \in X$, we have
\[
\widehat{\mathcal{O}}_{P} \cong k(P)[\![T]\!]
\]
\[
\widehat{K}_P \cong k(P)(\!(T)\!)
\]
for every $P \in X$. We want to establish a measure in
$\widehat{K}_P$ for $P \in X$. Take the class $\mathcal{A}_P$ of
subsets of $X$ which are bounded from below; i.e., $Z \in
\widehat{K}_P$ is said to be bounded from below if
``$\mathrm{ord}_T (x) \ge \mathrm{~constant}$'' for every $x \in
Z$.
\medskip

Let $n \in \mathbb{Z}$. Consider the map
\[
\pi_n:\widehat{K}_P \to k(P)(\!(T)\!)/(T^{n+1}).
\]

Notice that
\[
k(P)(\!(T)\!)/(T^{n+1}) \cong \left \{ \sum_{k=-c}^{n} a_k T^k
\mid a_k \in k(P) \right \} \cong k(P)^{n+c+1}
\]
where the latter isomorphism holds because the subsets are bounded
from below.
\medskip

We say that a subset $Z \subseteq \widehat{K}_P$, $Z \in
\mathcal{A}_P$ is measurable or cylindric if there exists $n \in
\mathbb{Z}$ such that $Z=\pi_n^{-1}(Y)$, for $Y \subseteq
k(P)(\!(T)\!)/(T^{n+1})$, $Y$ constructible. We define the measure
of such a $Z$ as
\[
\mu_P (Z):=[Y] \mathbb{L}^{-n}.
\]
We define the ring of ad\`eles $\mathbb{A}_X$ of $X$ as the
restricted product of the $\widehat{K}_P $'s with respect to the
$\widehat{\mathcal{O}}_P$'s, i.e., $x \in \mathbb{A}_X$ if and
only if $x=(x_P)_{P \in X}$, $x_P \in \widehat{\mathcal{O}}_P$ for
almost all $P$. Let $S$ be a finite subset of points. Define
\[
U_S:=Z \times \prod_{P \notin S} \widehat{\mathcal{O}}_P.
\]
We declare $U_S$ as an open subset, where $Z$ is a cylindric
subset of $\prod_{P \in S} \widehat{K}_P$, and take as measure
\[
\mu_S (Z):= \mu_S (\underline{Z}),
\]
where $\mu_S (\underline{Z})$ is defined as follows: let
\[
\underline{Z} \subseteq \prod_{P \in S} \widehat{K}_P
\]
with $\underline{Z}$ bounded from below (i.e., the order in $T$ of
each component is bounded from below); we say that $\underline{Z}$
is measurable or cylindric if there exists $\underline{n}=(n_P)_{P
\in S}$, $n_P \in \mathbb{Z}$ so that, for the map
\[
\pi_{\underline{n}}: \prod_{P \in S} \widehat{K}_P \to \prod_{P
\in S} k(P)(\!(T)\!) /(T^{n_P+1})
\]
we have $\underline{Z} = \pi^{-1}_{\underline{n}}
(\underline{Y})$, with $\underline{Y}$ a constructible subset of
\[
\prod_{P \in S} k(P)(\!(T)\!) /(T^{n_P+1}).
\]

The measure of $\underline{Z}$ is defined to be
\[
\mu (\underline{Z}):= \mu_S (\underline{Z}):=[\underline{Y}]
\mathbb{L}^{-\sum_{P \in S} n_P}.
\]

We declare $U_S$ as the open subsets and endow $\mathbb{A}_X$ with
a topology. Take the Borel $\sigma$-algebra generated by the
subsets $U_S$ which are bounded from below. Consider the maps
\[
\phi: \mathbb{A}_X \to \mathbb{Z}.
\]
The map $\phi$ is said to be cylindric if $\phi^{-1}(n) \subseteq
\mathbb{A}_X$ is a cylindric subset and $\phi$ is bounded from
below. Last we define the integral of a cylindric function $\phi:
\mathbb{A}_X \to \mathbb{Z}$ to be
\[
\int_{\mathbb{A}_X} T^{\phi} d \mu := \sum_{n \in \mathbb{Z}} \mu
(\phi^{-1}(n)) T^n,
\]
whenever the sum makes sense; in such a case the function $\phi$
is said to be integrable.
\medskip

We give now the projective versions of the above definitions. Let
$\mathbb{P}\widehat{K}_P$ be the projectivization of
$\widehat{K}_P$ with respect to the field $K$, that is,
$\mathbb{P}\widehat{K}_P:= (\widehat{K}_P \setminus \{0\}) /
\sim$, where the equivalence relation $\sim$ is defined as
follows: for every $a,b \in \widehat{K}_P$, we say that $a \sim b$
if and only if there exists $\lambda \in K \setminus \{0\}$ so
that $a = \lambda b$. From now on, all projectivizations we use
will be referred to $K$.
\medskip

Let $n \in \mathbb{Z}$. Let us consider the projectivization
$\mathbb{P} (k(P)(\!(T)\!)/(T^{n+1}))$ and let us adjoin one point,
that is, 
\[
\mathbb{P}^{\ast} (k(P)(\!(T)\!)/(T^{n+1})):=\mathbb{P}
(k(P)(\!(T)\!)/(T^{n+1})) \cup \{\ast\},
\]
with $\ast$ representing
the added point making sense the definition, so that the map
\[
\pi_n:\mathbb{P} \widehat{K}_P \to \mathbb{P}^{\ast}
(k(P)(\!(T)\!)/(T^{n+1}))
\]
is well-defined; i.e., for each $g \in \widehat{K}_P$, we denote
$[g]$ the class of $g$ in $\mathbb{P}\widehat{K}_P$. If $g \in
\widehat{K}_P \setminus (T^{n+1})$, then $\pi_n ([g]) \in
\mathbb{P} (k(P) (\!(T)\!)/(T^{n+1}))$, and if $g \in (T^{n+1})$,
then $\pi_n ([g]) \in \{ \ast \}$.
\medskip

Notice that if $Z \in \mathcal{A}_X$, then $\mathbb{P}Z \in
\mathcal{A}_X$, because
\[
\mathrm{ord}_T (x)= \mathrm{ord}_T (\lambda x)
\]
for all $x \in Z$ and all $\lambda \ne 0$. Thus, given a subset $Z
\subseteq \mathbb{P} \widehat{K}_P$, $Z \in \mathcal{A}_X$, $Z$ is
said to be cylindric or measurable if there exists $n \in
\mathbb{Z}$ so that
\[
Z = \pi_n^{-1} (Y)
\]
with $Y \subseteq \mathbb{P} k(P)(\!(T)\!)/(T^{n+1})$
constructible. Then we define the measure of such a $Z$ as
\[
\mu_P (Z):=[Y] \mathbb{L}^{-n}.
\]
Let us take now a finite subset of points $S$ of $X$ and define
\[
V_S = Z \times \mathbb{P} \prod_{P \notin S}
\widehat{\mathcal{O}}_P,
\]
where $Z$ is a cylindric subset of $\mathbb{P} \prod_{P \in S}
\widehat{K}_P$, and declare $V_S$ as an open subset of
$\mathbb{P}\mathbb{A}_X$, where $\mathbb{P}\mathbb{A}_X$ is the
projectivization of the ring of ad\`eles $\mathbb{A}_X$ of $X$. As
measure we take
\[
\mu_S (Z):= \mu_S (\underline{Z}),
\]
where $\mu_S (\underline{Z})$ is defined as follows. Let
$\underline{Z} \in \mathbb{P} \prod_{P \in S} \widehat{K}_P$,
$\underline{Z}=\mathbb{P} \underline{Z}^{\prime}$ with
$\underline{Z}^{\prime} \in \prod_{P \in S} \widehat{K}_P$ such
that the order function is bounded from below in every component.
We say that $\underline{Z}$ is measurable or cylindric if there
exists $\underline{n}=(n_P)_{P \in S}$, $n_P \in \mathbb{Z}$ such
that $\underline{Z} = \pi^{-1}_{\underline{n}} (\underline{Y})$
for
\[
\pi_{\underline{n}}: \mathbb{P} \left (\prod_{P \in S}
\widehat{K}_P \right ) \to \mathbb{P} \left (\prod_{P \in S}
k(P)(\!(T)\!) /(T^{n_P+1}) \right )
\]
and $\underline{Y}$ a constructible subset of $\mathbb{P} \left
(\prod_{P \in S} k(P)(\!(T)\!) /(T^{n_P+1}) \right )$. Hence we
define
\[
\mu_S (\underline{Z}):=\mu (\underline{Z}):=[\underline{Y}]
\mathbb{L}^{-\sum_{P \in S} n_P}.
\]

The projectivization $\mathbb{P} \mathbb{A}_X$ is endowed with the
topology inherited from $\mathbb{A}_X$.
\medskip

Analogous to the non projectivised case, a map
\[
\varphi: \mathbb{P}\mathbb{A}_X \to \mathbb{Z}
\]
is said to be cylindric if $\varphi^{-1}(n)
\subseteq \mathbb{P} \mathbb{A}_X$ is a cylindric subset and
$\varphi$ is bounded from below. Lastly, we define the integral of
a cylindric function $\varphi: \mathbb{P} \mathbb{A}_X \to
\mathbb{Z}$ to be
\[
\int_{\mathbb{P} \mathbb{A}_X} T^{\varphi} d \mu := \sum_{n \in
\mathbb{Z}} \mu (\varphi^{-1}(n)) T^n,
\]
whenever the sum makes sense; in such a case the function
$\varphi$ is said to be integrable.
\medskip

Notice that the function $\phi: \mathbb{A}_X \to \mathbb{Z}$ is cylindric if and only if the function $\varphi: \mathbb{P}\mathbb{A}_X \to \mathbb{Z}$ is cylindric, and we have
\[
(\mathbb{L}-1) \int_{\mathbb{P}\mathbb{A}_X} T^{\varphi} d \mu = \int_{\mathbb{A}_X} T^{\phi} d \mu.
\]
(cf. \cite[Remark 4]{MZ}).

Let us take the cylindric function $T^{\mid \underline{v}(\cdot)\mid}:\oo \to \mathbb{Z}[\![T]\!]$ defined by $\underline{z} \mapsto T^{\left\Vert \underline{v}(\underline{z})  \right\Vert}$, with $T^{\left\Vert \underline{v}(\underline{z})  \right\Vert}:=0$ if $\underline{v}(\underline{z})=\infty$.

\begin{Defi}
The integral
\[
\zeta (X, T):= \int_{\mathbb{A}_X} T^{\left\Vert \underline{v}(\underline{z})  \right\Vert} d \mu 
\]
will be called the ad\'elic zeta function associated with $X$.
\end{Defi}

We want now to show that this global zeta function decomposes into the product of the ad\'elic zeta function of the normalisation with the local zeta functions corresponding to the singular points of the curve, which generalises slightly a result of Z\'{u}\~{n}iga (cf. \cite[Corollary 2.2]{Z3}).
\medskip

Let $P \in S=\{P_1, \ldots , P_r\}$. Let us write $Z_P = \pi^{-1}_{n_P}(Y_P)$ for $Y_P$ a constructible subset of $k(P)(\!(T)\!) /(T^{n_P+1})$. Then one has
\begin{eqnarray}
\mu (\underline{Z}) & = & [Y_{P_1} \times \ldots \times Y_{P_r}] \mathbb{L}^{-n_{P_1}- \ldots - n_{P_r}} = [Y_{P_1}] \mathbb{L}^{-n_{P_1}} \cdot \ldots \cdot [Y_{P_r}] \mathbb{L}^{-n_{P_r}}   \nonumber \\
&=& \mu (Z_{P_1}) \cdot \ldots \cdot \mu (Z_{P_r}).  \nonumber
\end{eqnarray}

Denoting $\phi_i$ the restriction $\phi \mid_{\widehat{\oo}_{P_i}}: \widehat{O}_{P_i} \to \mathbb{Z}$, then we have
\[
\int_{Z} T^{\phi} d \mu = \int_{\widehat{\oo}_{P_1}} T^{\phi_1} d \mu \cdots \ldots \cdot \int_{\widehat{\oo}_{P_r}} T^{\phi_r} d \mu.
\]

Because of the multiplicativity of the measure $\mu (\cdot)$ one has 
\[
\mu (\mathbb{A}_X)=\mu (Z \times \prod_{P \notin S} \oo_P) = \mu (Z) \mu (\prod_{P \notin S} \oo_P),
\]
hence
\begin{eqnarray}
\zeta (X,T)&=&\int_{Z \times \prod_{P \notin S} \oo_P} T^{\phi} d \mu = \int_{Z} T^{\phi} d \mu \cdot \int_{\prod_{P \notin S} \oo_P} T^{\phi} d \mu \nonumber \\
&=&  \int_{Z} T^{\phi} d \mu \cdot \int_{\widetilde{X}} T^{\phi} d \mu. \nonumber
\end{eqnarray}

Furthermore, by \cite[Corollary 4]{MZ} one has
\[
Z(\oo_P,T)= \frac{1}{(\mathbb{L}-1)\mathbb{L}^{-\delta_P -1}} \int_{\oo_P} T^{\left\Vert \underline{v}(\underline{z})  \right\Vert} d \mu.
\]
We have thus proved:

\begin{Theo} \label{Theo:main}
Let $\delta:=\delta_{P_1} + \ldots + \delta_{P_r}$. We have
\[
\mathbb{L}^{\delta} \zeta(X, T)=(1-\mathbb{L}^{-1})^r \cdot \zeta(\widetilde{X},T) \cdot \prod_{i=1}^{r}Z(T,\oo_{P_i}).
\]
\end{Theo}

By taking reduction modulo $\mathfrak{P}$ we obtain:

\begin{Cor} \label{cor:penultimo}
Let $X^{\prime}:= X \text{ mod } \mathfrak{P}$. We have:
\[
\mathbb{L}^{\delta}  \zeta(X^{\prime}, T)= (1-\mathbb{L}^{-1})^r \cdot \zeta(\widetilde{X^{\prime}},T) \cdot \prod_{i=1}^{r} Z(T,\oo_{P_i, X^{\prime}}).
\]
\end{Cor}

Theorem \ref{Theo:main} extends to the result of \cite{Z3} if one specialises $[\cdot ]$ to the additive invariant $\sharp: \mathrm{Var}_k \to \mathbb{Z}$ given by counting points by taking a finite field $\mathbb{F}_q$ as ground field. 
\medskip

Notice that our definition differs from the one introduced by Z\'{u}\~{n}iga in \cite{Z3}; he considered there a Dirichlet series $Z(\mathrm{Ca}(X),T)$ associated to the effective Cartier divisors on the algebraic curve defined over $\mathbb{F}_q$. Both series are related by means of the equality
\[
\sharp (\zeta(X, T)) =\frac{(1-q^{-1})^r}{q^{\delta}} Z(\mathrm{Ca}(X),T).
\]

We can express the formula given by Corollary \ref{cor:penultimo} in another nice way.
Let $P$ be a point on $X$. Let us consider the group $U_{\widetilde{\oo}_P}$ of the units of the normalisation of the local ring $\oo_P$. The subgroup $U_{\oo_P}$ of the group $U_{\widetilde{\oo}_P}$ is of finite index, say $(U_{\widetilde{\oo}_P}:U_{\oo_P})$. In fact, by taking the normalisation morphism $\pi:\widetilde{X} \to X$, one has $\pi^{-1}(P)=\{Q_1, \ldots , Q_m\}$, where $Q_i$ are the branches of $X$ centered at $\oo_P$ and correspond to maximal ideals $\mathfrak{m}_i$ of the semilocal ring $\widetilde{\oo_P}$; furthermore the following well-known equality holds (cf. \cite[p.182]{stohr}):
\[
(U_{\widetilde{\oo}_P}:U_{\oo_P})=\frac{q^{\delta_P}}{(1-q^{-r_P})} \prod_{i=1}^{m} (1-q^{-\mathrm{deg}(Q_i)}),
\]
where $r_P$ is the degree of the residue field of $\oo_P$ over $\mathbb{F}_q$, and the degree $\mathrm{deg}(Q_i):=\dim \widetilde{\oo}_P/\mathfrak{m}_i$, for every $i \in \{1, \ldots , m\}$. Since we are assuming the local rings $\oo_P$ to be totally rational, both the integers $r_P$ and $\mathrm{deg}(Q_i)$, for all $i \in \{1, \ldots , m\}$, are equal to $1$. By applying Corollary \ref{cor:penultimo} we get:
\begin{Cor}
\[
\sharp (\zeta(X, T))= \sharp (\zeta(\widetilde{X},T)) \cdot \prod_{P \in S} \prod_{Q \in \pi^{-1}(P)} \frac{1-q^{-1}}{(U_{\widetilde{\oo}_P}:U_{\oo_P})} \cdot Z(T,\oo_{P_i}).
\]
\end{Cor}

\end{document}